\documentclass[]{article} 
\usepackage[top=3cm,bottom=3cm,left=3cm,right=3cm]{geometry}
\usepackage{parskip}
\usepackage{bm}
\usepackage[scr=boondox]{mathalfa}
\usepackage{amsfonts}
\usepackage{amsmath}
\usepackage{amsthm}
\usepackage{amssymb}
\usepackage{dsfont}
\usepackage[colorlinks=true,linkcolor=blue,citecolor=cyan,urlcolor=cyan]{hyperref}
\usepackage[nameinlink]{cleveref}
\crefformat{equation}{(#2#1#3)}

\usepackage{pgf,tikz}
\usetikzlibrary{arrows,shapes,positioning}
\usetikzlibrary{decorations.pathmorphing}
\usetikzlibrary{patterns,decorations.pathreplacing}
\usetikzlibrary{patterns}

\newcommand{\R}{\mathds{R}}
\newcommand{\Z}{\mathds{Z}}
\newcommand{\N}{\mathds{N}}

\renewcommand{\P}[1]{\mathds{P}\left(#1\right)}
\newcommand{\ind}{\mathds{1}}
\newcommand{\E}[1]{\mathds{E}\left[#1\right]}
\newcommand{\itg}{\displaystyle\int}

\newcommand{\cond}{\ \middle| \ }
\newcommand{\M}{\bm{M}}
\newcommand{\F}{\mathcal{F}}
\newcommand{\Ge}[2]{\bm{\Gel(} #1,#2 \bm{)} }
\newcommand{\T}{\mathcal{T}}
\newcommand{\U}{\mathcal{U}}

\newcommand{\Pois}{\mathscr{P}\!\mathscr{o}\!\mathscr{i}\!\mathscr{s}}

\newcommand{\Vvert}{\vert \hspace{-1pt} \vert \hspace{-1pt} \vert}

\DeclareMathOperator{\card}{Card}
\DeclareMathOperator{\spr}{SpR}
\DeclareMathOperator{\Gel}{Ge}
\DeclareMathOperator{\tp}{tp}
\DeclareMathOperator{\bd}{bd}

\newtheorem{theorem}{Theorem}

\newtheorem{lemma}{Lemma}

\newtheorem{proposition}{Proposition}

\newtheorem{definition}{Definition}

\newtheorem{assumption}{Assumption}

\newtheorem{remark}{Remark}

\title{Exponential moments for Hawkes processes under minimal assumptions} 

\author{Th\'eo~Leblanc \footnote{Universit\'e Paris Dauphine PSL, France. Email : theo.leblanc@dauphine.psl.eu }}

\date{}

\begin{document}

\maketitle

\begin{center}
    \textbf{Abstract}
    
    \begin{minipage}{0.8\linewidth}
    We prove that the number of points of a stationary linear Hawkes process lying in any bounded subset of the real line has exponential moments, without any other assumption than the one needed for existence of such stationary process, namely the spectral radius of the matrix of ${\mathbb L}^1$ norms of interaction functions is smaller than one. The proof relies on a mass transport principle argument. We also specify the dependence of the bounds with respect to the base rates and the matrix of ${\mathbb L}^1$ norms of interaction functions defining the Hawkes process and give a functional version of the result.\\

    \textbf{Keywords:} Hawkes processes; exponential moments; cluster representation.\\
    \textbf{MSC classifications:} 60G55; 60E15; 60J85.
\end{minipage}
\end{center}

\section{Introduction}

\subsection{Motivation}

Hawkes processes, introduced by Hawkes in 1971 in \cite{Hawkes71}, are modelling the continuous time arrivals of some random $0-1$ events when the arrivals depend on the past events. Hawkes processes can be self exciting or self inhibiting. This flexibility allows Hawkes processes to model a large variety of phenomena, from the time occurrences of earthquakes \cite{Otaga} to the activity of the neural network of the brain \cite{GalvesLoc, AgeDepHP}, to cite but of few. Existence of stationary linear Hawkes processes, which only models excitation, is guaranteed under the well known condition that the spectral radius of the matrix of ${\mathbb L}^1$ norms of interaction functions is smaller than one \cite{BréMas}. This condition is expressed in \Cref{A1} in \Cref{setting}. From a statistical point of view, having controls on the moments of Hawkes processes is key to construct robust estimation procedures of the interaction parameters of the process. Exponential moments of stationary Hawkes processes have been proved under conditions such as bounded memory interactions \cite{LassoHawkes,RBRoy} or decay assumptions on the interaction function in the univariate case \cite{ROUEFF20161710}. But, up to our knowledge, there are no results about the existence of exponential moments for Hawkes processes under the minimal spectral radius condition that guarantees the existence of stationary Hawkes processes. In this paper we prove that a stationary Hawkes process has exponential moments under the spectral radius assumption, regardless of the tails of the interaction functions. We also provide quantitative bound of the exponential moments.

Along this paper we denote $\R$ the set of real numbers, $\N$ the set of integers: $\N=\{0,1,2,\cdots\}$ and $\N^*=\N\setminus\{0\}$. 

\subsection{Setting}\label{setting}

A \textit{point process} on the real line $\R$ is a random set of points $N\subset \R$ such that almost surely, $N$ is locally finite. Given $\M=[\![1,M]\!]$ with $M$ a positive integer, a $\M$-\textit{multivariate point process} on $\R$ is the collection $(N^m)_{m\in\M}$ of random locally finite subsets of $\R$. Under mild conditions, a multivariate point process $(N^1,\cdots,N^M)$ is characterised by its stochastic intensity process $(\lambda^1,\cdots,\lambda^M)$, defined by
\begin{equation*}
    \P{N^m \ \text{has a point in} \ [t,t+dt] \cond \F_t} = \lambda^m_t dt, \quad m=1,\cdots,M,
\end{equation*}
where $(\F_t)_{t\in\R}$ is the filtration generated by $(N^1,\cdots,N^M)$. See \cite{PPQ} for more precise statements. Hawkes processes are specific point processes. To define them, we consider $\bm{\mu}:=(\mu_m)_{m\in\M} \subset \R_+^{\M}$ and $\bm{h} := (h^m_{m'})_{m,m'\in\M}$, with $h^m_{m'}:\R_+\longrightarrow\R_+$ measurable functions.

\begin{definition}\label{def hawkes}
    Let $(\Omega,\F,\mathds{P})$ a probabilistic space and a filtration $(\F_t)_{t\in\R}$. A Hawkes process with parameters $\M,\bm{\mu},\bm{h}$ is a multivariate point process $\bm{N}:=(N^m)_{m\in\M}$ such that
    \begin{itemize}
        \item $\bm{N}$ is adapted to $(\F_t)_{t\in\R}$,
        \item almost surely, for $m\neq m'$ in $\M$, $N^m$ has no common points with $N^{m'}$,
        \item $\bm{N}$ has a predictable intensity $\bm{\lambda}:=(\lambda^m)_{m\in\M}$ given by
        \begin{equation}\label{eq hawkes}
        \lambda^m_t = \mu_m + \sum_{m'\in M} \int_{-\infty}^{t^-} h^m_{m'}(t-s)dN^{m'}_s \quad m\in\M, \ t\in\R.
        \end{equation}
    \end{itemize}
\end{definition}

In the previous definition, the constants $\mu_m$ are called \textit{base rates}, modelling the spontaneous rates of arrivals of points. The functions $h^m_{m'}$ that take into account the influence of each point process on the others are called \textit{interaction functions}.

In \cite{BréMas}, Brémaud and Massoulié proved in particular that such a linear Hawkes process, in the sense of \Cref{def hawkes} exists and is stationary under the following assumption.

\begin{assumption}\label{A1}
    The spectral radius of the matrix $\bm{H}=(\Vert h^m_{m'} \Vert_1)_{m',m\in\M}$ is strictly smaller than $1$, ie $\spr(\bm{H})<1$.
\end{assumption}

As explained in the next section, this assumption is sufficient to have exponential moments.

\subsection{Main result}

Given a Hawkes process $\bm{N}$ and $B\subset\R$ we are interested in the number of points lying in $B$:
\begin{equation*}
    N(B) = \sum_{m\in\M} N^m(B)
\end{equation*}
where $N^m(B) := \card( N^m \cap B)$. To state our main result we need the following definition.

\begin{definition}\label{GE}
    Let $A$ a square matrix. We say that $A$ satisfies the property $\Ge{r}{K}$ with $r,K\geq 0$ and denote $A\in\Ge{r}{K}$ if
    \begin{equation*}
        \forall n\in\N^*, \ \Vvert A^n \Vvert_{\infty} \leq Kr^n,
    \end{equation*}
    where and $\Vvert \cdot \Vvert_{\infty}$ is the operator norm associated to the norm $\Vert \cdot \Vert_{\infty}$.
\end{definition}

The well known Gelfand's Theorem states that for any square matrix $A$, for any norm $\Vert \cdot \Vert$, we have $\Vert A^n \Vert ^{1/n} \xrightarrow[n\to\infty]{} \spr(A)$. Thus, for any $r>\spr(A)$ there exists $K<\infty$ such that $A\in \Ge{r}{K}$. It follows that under \Cref{A1}, there exist $r$ and $K$ with $\spr(\bm{H}) < r < 1$ and $0< K<\infty$ such that $\bm{H}\in \Ge{r}{K}$. Now we can state the main result of this paper.

\begin{theorem}\label{moment exp}
    Suppose that \Cref{A1} holds, ie the spectral radius of $\bm{H} = (\Vert h^m_{m'} \Vert_1)_{m',m}$ is strictly less than $1$. Then there exists $\xi>0$ depending only on $\bm{H}$ such that for all $B\subset \R$ bounded, we have
    \[ \E{e^{\xi N(B)}} <\infty.\]
    More precisely, if $B\subset [s,s+L)$ for some $s\in\R$ and $L>0$, and if $\bm{H}\in\Ge{r}{K}$ with $0< r< 1$ and $0\leq K<\infty$ then we have
    \begin{equation}\label{borne moment exp} \E{e^{u \xi_{r,K} N(B)}} \leq \exp\left( \left[\Big(\dfrac{1+r}{2r}\Big)^{u} -1\right]L \sum_{m\in M} \mu_m \right), \quad \text{for all} \ u\in[0,1]
    \end{equation}
    with 
    \begin{align*}
        &\xi_{r,K} = \dfrac{\log\Big(\dfrac{1+r}{2r}\Big)}{1+\dfrac{2K}{1-r}}.
    \end{align*}
\end{theorem}

Some remarks are in order.

The bound given in \cref{borne moment exp} can be rewritten as follows:
\begin{equation}\label{borne moment exp RW} \E{e^{t N(B)}} \leq \exp\left( \Big[e^{t C_{r,K}}-1\Big] L \sum_{m\in M} \mu_m \right), \quad 0\leq t\leq \xi_{r,K}
\end{equation}
with $C_{r,K} = 1+\dfrac{2K}{1-r}$. Thus, the bound for $\E{e^{t N(B)}}$ is increasing with $K$ and $r$, which is the expected behaviour. 

In the univariate case, ie $\card(\M)=1$, the matrix $\bm{H}$ is in fact a scalar with value $\alpha = \Vert h \Vert_1$ where $h$ is the (only one) interaction function. Since $\bm{H}^n = (\alpha^n)$ it is straightforward that $\bm{H}\in \Ge{\alpha}{1}$. 
If $\bm{H}=0$, meaning that there are no interactions and the Hawkes process consists of independent homogeneous Poisson processes with rates $\mu_m$ for $m\in\M$, then we have $\bm{H}\in\Ge{r}{0}$ for all $r\in(0,1)$. In this case, the exponential moments of $N([0,L])$ can be explicitly calculated and compared to the bound \cref{borne moment exp} as follows:
\[ \exp\left( (e^t-1) L \sum_{m\in M} \mu_m \right) = \E{e^{tN([0,L])}} \underset{\text{by} \ \cref{borne moment exp}}{\leq} \exp\left( (e^{t}-1) L \sum_{m\in M} \mu_m \right) \quad t\geq0.\]
Indeed, $C_{r,0} = 1$ and $\xi_{r,0} \xrightarrow[r\to 0^+]{} +\infty$. Thus, in this particular case, the bound \cref{borne moment exp} becomes an equality.

The dependency on $M$, the cardinal of $\M$, is hidden in the constant $K$ of $\Ge{r}{K}$. Indeed, the operator norm $\Vvert \cdot \Vvert_{\infty}$ takes into account the dimension of the matrix in the following sense: for a square matrix $A$ of size $d\geq 1$, we have
\[ \Vvert A \Vvert_{\infty} = \max_{i \in [\![1,d]\!]} \sum_{j=1}^d \vert A_{ij} \vert.\]

\begin{remark}
    The constant $\xi_{r,K}$, which comes from \Cref{moment exp GW} which states exponential moments of Galton Watson processes with Poisson offspring measure, may be not optimal. In particular, in the simplest univariate case of a Galton Watson process $\T$ with offspring measure $\Pois(\alpha)$ where $0\leq \alpha<1$, if we denote $g(\xi) := \log \E{e^{\xi \card(\T)}}$ then for $\xi>0$, $g(\xi)<\infty$ if and only if the equation $x=\xi+\alpha(e^x-1)$ has a positive solution for $x$ and in this case $g(\xi)$ is given by the first positive solution (on $x$) of $x=\xi+\alpha(e^x-1)$. Given $\alpha$, finding the largest $\xi$ such that there exists a positive solution of the above equation cannot be solved with the usual functions. Thus the optimal $\xi_{r,K}$ can not be expressed in closed form. However, the above reasoning shows (with the proof of \Cref{moment exp}) that for each matrix $\bm{H}$ such that $\spr(\bm{H})<1$, there exists a $\xi_{\bm{H}}>0$ satisfying $\E{e^{\xi_{\bm{H}} N(B)}}<\infty$ for all $B$ bounded and $\E{e^{\xi N(B)}} = \infty$ for any $\xi>\xi_{\bm{H}}$ and $B$ non-zero interval.
\end{remark}

\subsection{Functional result}

Let $f$ a non negative real function and define 
\[N(f) := \sum_{m\in \M} \sum_{x\in N^m} f(x).\]
The following result states exponential moments of $N(f)$.

\begin{theorem}\label{functional}
    Let $f$ a non negative real function, and $T>0$. Suppose that $\bm{H}\in\Ge{r}{K}$ with $0<r<1$ and $0\leq K<\infty$. Then we have
    \[\E{e^{\xi N(f)}} \leq \exp\left( \bigg[ e^{\xi \vert f \vert^{\infty}_{1,T} C_{r,K}} -1\bigg] T \sum_{m\in M} \mu_m \right), \quad 0\leq \xi \leq \xi_{r,K} / \vert f \vert^{\infty}_{1,T}\]
    where $\xi_{r,K}$ and $C_{r,K}$ are defined in \Cref{moment exp} and 
    \[ \vert f \vert^{\infty}_{1,T} =  \sup_{t\in\R} \sum_{n\in\Z} f(t+nT).\]
\end{theorem}

\Cref{functional} is an extension of \Cref{moment exp}. Indeed to recover the classical exponential moments, one can take $f=\ind_B$ and $T=L$.

The $\vert \cdot \vert^{\infty}_{1,T}$ is similar to a discrete $L^1$ norm for non negative functions with discretization scale $T$.

The proof of \Cref{functional} is similar to the proof of \Cref{moment exp} and is based on the same arguments. Thus we only give the key details of the proof.

Before giving the proof of \Cref{moment exp}, we recall in the next section the cluster representation of Hawkes processes, which is crucial to prove our result.

\section{Cluster representation}

The \textit{cluster representation}, introduced by Hawkes and Oakes in 1974 \cite{HawkesOakes}, is a powerful tool to study Hawkes processes. In a few words, the cluster representation means that, under \Cref{A1}, a Hawkes process as introduced in \Cref{def hawkes} can be seen as a branching Poisson process on homogeneous Poisson processes with rates $\mu_m$ for $m\in\M$. Points generated by the homogeneous Poisson processes are called immigrants points. Points generated by the branching Poisson processes are called offspring points. Each immigrant point may generate new points, called its children, due to the self exciting term of \cref{eq hawkes}. Then each child may also have some children and so on. These successive arrivals of offspring points, with its genealogy, is called a cluster. In the sequel we define more formally the clusters.

We use the Ulam Harris Neveu notation to label the individuals of trees. Let $\U = \cup_{n\in\N} (\N^*)^n$ with $(\N^*)^0=\{\varnothing\}$. A $\Pois(\bm{h})$-cluster with root of type $m_0\in\M$ and born at time $t_0\in\R$ is defined by $\big((u,\tp(u),\bd(u))\big)_{u\in\T}$ where
\begin{itemize}
    \item $\T\subset \U$ is a random tree, for $u\in\T$, $\tp(u)\in\M$ is the type of $u$ and $\bd(u)\in\R$ is the birth date of $u$.
    \item $\big((u,\tp(u))\big)_{u\in\T}$ is a $\Pois(\bm{H})$-Galton Watson tree with root $\varnothing$ of type $m_0$ (ie $\tp(\varnothing)=m_0$). It means that independently for each point $u\in\T$, if we denote $X(u,\tp(u),m)$ the number of children of $u$ with type $m\in\M$, then for $m\in\M$ the random variables $X(u,\tp(u),m)$ are independent and $X(u,\tp(u),m) \sim \Pois(\bm{H}_{\tp(u),m})$ for all $m\in\M$.
    \item Conditionally to $\big((u,\tp(u))\big)_{u\in\T}$, the birth dates are distributed as follows. For the root of $\T$ we have $\bd(\varnothing)=t_0$. Then, independently, for any $u,v\in\T$ with $v$ a child of $u$, we have
    \[\bd(v)-\bd(u) \sim \frac{h^{\tp(v)}_{\tp(u)}(s)}{ \itg h^{\tp(v)}_{\tp(u)}(t)dt} \ind_{s>0}ds.\]
\end{itemize}

Finally, let $\pi^m$ be independent Poisson processes on $\R$ with rates $\mu_m$ for $m\in\M$. They are the \textit{immigrant points}. Conditionally to $(\pi^m)_{m\in\M}$, let $G^m_x$ be independent $\Pois(\bm{h})$-clusters with root of type $m$ and born at time $x$ for $m\in\M$ and $x\in\pi^m$. The following result is known \cite{HawkesOakes} and provides the cluster representation for Hawkes processes.

\begin{proposition}\label{cluster rep}
    Under \Cref{A1}, if $(N^m)_{m\in\M}$ is a Hawkes process from \Cref{def hawkes}, then we have the following equality in distribution
    \begin{equation*}
        (N^m)_{m\in\M} = \big(\{ s\in\R \mid \exists m'\in\M, \exists x\in\pi^{m'}, \exists u\in \U \ \text{such that} \ (u,m,s)\in G^{m'}_x\}\big)_{m\in\M}
    \end{equation*}
    where the immigrants $(\pi^m)_{m\in\M}$ are independent Poisson point processes with rates $\bm{\mu}$.
    Equivalently, in distribution, $N^m$ is the collection of all the birth dates of any points of type $m$ of any clusters generated by the immigrants points. 
\end{proposition}

\Cref{fig:cluster rep} illustrates the cluster representation.

\begin{figure}[htbp]
\begin{center}
\vspace{-1.5cm}
\begin{tikzpicture}[scale=0.7]

\draw[line width=1pt, fill=blue!50, shift={(10,0)}, rotate=-35, scale=1.8] (0,0) ..controls +(-1,2) and +(1,2).. (0,0);
\draw[line width=1pt, fill=blue!50, shift={(7,0)}, rotate=-27, scale=3] (0,0) ..controls +(-1,2) and +(1,2).. (0,0);
\draw[line width=1pt, fill=blue!50, shift={(6.6,0)}, rotate=-30, scale=1.5] (0,0) ..controls +(-1,2) and +(1,2).. (0,0);
\draw[line width=1pt, fill=blue!50, shift={(5.8,0)}, rotate=-50, scale=1] (0,0) ..controls +(-1,2) and +(1,2).. (0,0);
\draw[line width=1pt, fill=blue!50, shift={(4,0)}, rotate=-29, scale=0.5] (0,0) ..controls +(-1,2) and +(1,2).. (0,0);
\draw[line width=1pt, fill=blue!50, shift={(3,0)}, rotate=-27, scale=2.5] (0,0) ..controls +(-1,2) and +(1,2).. (0,0);
\draw[line width=1pt, fill=blue!50, shift={(1.5,0)}, rotate=-40, scale=2] (0,0) ..controls +(-1,2) and +(1,2).. (0,0);

\draw[line width=1.5pt, -stealth] (0,0)--(12.8,0);
\draw (13,-0.4) node {time};

\tikzstyle{every node}=[circle, draw, fill=red, inner sep=0pt, minimum width=6pt]

\draw (1.5,0) node {};
\draw (3,0) node {};
\draw (4,0) node {};
\draw (5.8,0) node {};
\draw (6.6,0) node {};
\draw (7,0) node {};
\draw (10,0) node {};

\end{tikzpicture}
\end{center}
\caption{Illustration of the cluster representation. The red dots are the immigrant points. They arrive according to the immigrant rate, being $\mu_m$ for type $m\in\M$, and give birth to clusters, represented by the blue shapes. In the end, only the birth dates and the types are relevant for Hawkes processes, which can be interpreted as a projection of the blue shapes on the time axis, in the sense that the tree structure is forgotten}
\label{fig:cluster rep}
\end{figure}
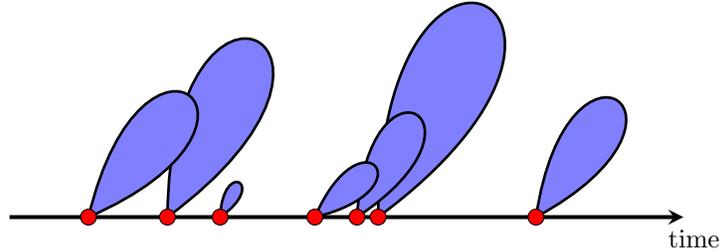

\section[Proof of Theorem 1.4]{Proof of \Cref{moment exp}}

In the sequel since \Cref{A1} is satisfied we fix $r\in(0,1)$ and $K\geq 0$ such that $\bm{H}\in\Ge{r}{K}$. Thanks to \Cref{cluster rep} we can express a Hawkes process in terms of clusters.

Without loss of generality, due to stationarity, we can suppose that $B=[0,L)$. Let $\xi = u\xi_{r,K}$ with $u\in[0,1]$ and $\xi_{r,K}$ defined in \Cref{moment exp}. Let $(\pi^m)_{m\in\M}$ be the immigrant points. For $x\in\pi^m$ let $G^m_x$ be the cluster associated to $x$. For any cluster $G$ and $A\subset \R$ we define
\[G(A) = \card\big( \{ s\in A \mid \exists (u,m)\in \U\times \M, \ (u,m,s)\in G\}\big).\]
Then we have from \Cref{cluster rep}
\[ N(B) = \sum_{m\in \M}  \sum_{x\in \pi^m} G^m_x([0,L)).\]
Until \cref{eq dem th1 1} the proof is an application of \cite{daley2003introduction} Example 8.3(c). For sake of completeness we give the few lines of calculation since it contains a step which is crucial for the proof. Since conditionally to the immigrants the clusters are independent, we have
\begin{align*}
\E{e^{\xi N(B)}} & = \E{\prod_{m\in \M} \prod_{x\in \pi^m} \E{e^{\xi G^m_x([0,L))} \cond (\pi^m)_{m\in\M}}}.
\end{align*}
Our proof relies on the following change of paradigm, which will lead to a mass transport principle \cite{Khezeli}, see \Cref{fig:mass transport 1}. For $m\in \M$, let $G^m$ a random variable with the law of a cluster with root of type $m$ and born at $t=0$. Then, since the clusters are invariant by time translation, we have for $m\in \M$ and $x\in\pi^m$,
\begin{equation}\label{chgt de temps}
    \E{e^{\xi G^m_x([0,L))} \cond (\pi^m)_{m\in\M}} = \E{e^{\xi G^m([-x,-x+L))}}=:g(-x,m,\xi,L).
\end{equation} 
The left hand side of the formula represents the quantity that immigrant $x$ gives to $B=[0,L)$. The right hand side represents what $B-x$ receives from an immigrant at $t=0$, see \Cref{fig:mass transport 1}. 

\begin{figure}[htbp]
\begin{center}
\hfill
\begin{minipage}[c]{0.45\linewidth}
\centering
\begin{tikzpicture}[scale=0.45]

\clip (0,-1.5) rectangle (13.5,7);

\draw[line width=1pt, fill=blue!50, shift={(8.5,0)}, rotate=-50, scale=1.5] (0,0) ..controls +(-1,2) and +(1,2).. (0,0);
\draw[line width=1pt, fill=blue!50, shift={(5.5,0)}, rotate=-68, scale=5.5] (0,0) ..controls +(-1,2) and +(0.2,2).. (0,0);
\draw[line width=1pt, fill=blue!50, shift={(2,0)}, rotate=-60, scale=6.6] (0,0) ..controls +(-0.5,2) and +(0.3,2).. (0,0);

\draw[line width=1.5pt, -stealth] (0,0)--(13.5,0);

\draw (2,-1) node {$x_{3}$};
\draw (5.5,-1) node {$x_{2}$};
\draw (8.5,-1) node {$x_{1}$};
\draw (9.5,-1) node {$0$};
\draw (11,-1) node {$L$};

\definecolor{aqua}{rgb}{0.0,1.0,1.0}
\begin{scope}
    \clip (9.5,0) rectangle (11,6);
    \fill[pattern={north west lines}] (9.5,0) rectangle (11,6);
    \draw[line width=1pt, fill=aqua, shift={(8.5,0)}, rotate=-50, scale=1.5] (0,0) ..controls +(-1,2) and +(1,2).. (0,0);
    \draw[line width=1pt, fill=aqua, shift={(5.5,0)}, rotate=-68, scale=5.5] (0,0) ..controls +(-1,2) and +(0.2,2).. (0,0);
    \draw[line width=1pt, fill=aqua, shift={(2,0)}, rotate=-60, scale=6.6] (0,0) ..controls +(-0.5,2) and +(0.3,2).. (0,0);
\end{scope}
    
\tikzstyle{every node}=[circle, draw, fill=red, inner sep=0pt, minimum width=6pt]

\draw[dashed] (9.5,0)--(9.5,6);
\draw[dashed] (11,0)--(11,6);

\draw (2,0) node {};
\draw (5.5,0) node {};
\draw (8.5,0) node {};

\end{tikzpicture}
    
\end{minipage}
\hfill
\begin{minipage}[c]{0.45\linewidth}
\centering
    
\begin{tikzpicture}[scale=0.45]

\clip (0,-1.5) rectangle (13.5,7);

\draw[line width=1pt, fill=blue!50, shift={(2,0)}, rotate=-60, scale=6.8] (0,0) ..controls +(-0.6,2) and +(0.5,2).. (0,0);

\draw[line width=1.5pt, -stealth] (0,0)--(13.5,0);

\draw (2,-1) node {$0$};
\draw (3,-1) node {$-x_{1}$};
\draw (6,-1) node {$-x_{2}$};
\draw (9.5,-1) node {$-x_3$};

\definecolor{aqua}{rgb}{0.0,1.0,1.0}
    
\begin{scope}
    \clip (9.5,0) rectangle (11,6);
    \fill[pattern={north west lines}] (9.5,0) rectangle (11,6);
    \draw[line width=1pt, fill=aqua, shift={(2,0)}, rotate=-60, scale=6.8] (0,0) ..controls +(-0.6,2) and +(0.5,2).. (0,0);
\end{scope}

\begin{scope}
    \clip (6,0) rectangle (7.5,6);
    \fill[pattern={north west lines}] (6,0) rectangle (7.5,6);
    \draw[line width=1pt, fill=aqua, shift={(2,0)}, rotate=-60, scale=6.8] (0,0) ..controls +(-0.6,2) and +(0.5,2).. (0,0);
\end{scope}

\begin{scope}
    \clip (3,0) rectangle (4.5,6);
    \fill[pattern={north west lines}] (3,0) rectangle (4.5,6);
    \draw[line width=1pt, fill=aqua, shift={(2,0)}, rotate=-60, scale=6.8] (0,0) ..controls +(-0.6,2) and +(0.5,2).. (0,0);
\end{scope}

\draw[dashed] (9.5,0)--(9.5,6);
\draw[dashed] (11,0)--(11,6);

\draw[dashed] (6,0)--(6,6);
\draw[dashed] (7.5,0)--(7.5,6);

\draw[dashed] (3,0)--(3,6);
\draw[dashed] (4.5,0)--(4.5,6);

\tikzstyle{every node}=[circle, draw, fill=red, inner sep=0pt, minimum width=6pt]

\draw (2,0) node {};

\end{tikzpicture}
    
\end{minipage}
\end{center}
    
\caption{Illustration of the mass transport principle. In expectation, both side are equal.}
\label{fig:mass transport 1}
\end{figure}
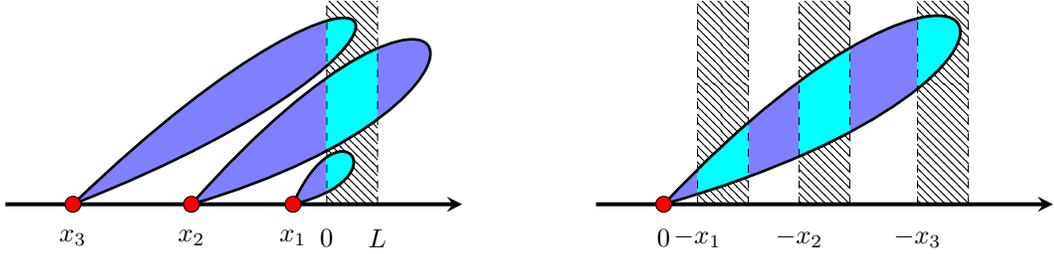
    
We can now use Campbell Theorem for Poisson processes (see \cite{PPK}, section 3.2):
\begin{align}
    \E{e^{\xi N(B)}} & = \E{\prod_{m\in \M}  \prod_{x\in \pi^m} g(-x,m,\xi,L) } \nonumber\\
    & = \exp\left( \sum_{m\in \M} \itg_{-\infty}^L \big[g(-t,m,\xi,L)-1\big] \mu_m dt\right)\nonumber\\
    & = \exp\left( \sum_{m\in \M} \itg_{-L}^{\infty} \big[g(t,m,\xi,L)-1\big] \mu_m dt\right).\label{eq dem th1 1}
\end{align}
It remains to prove that $L^{-1}\itg_{-L}^{\infty} \big[g(t,m,\xi,L)-1\big]dt$ is finite and more precisely, bounded by $((1+r)/2r)^{u}-1$. For this purpose, in the literature, it is classical to use assumptions on the tail of the interaction functions, such as bounded support assumptions, to control the integrand for each $t$ to prove that the integral is convergent, see \cite{LassoHawkes}. Here we will prove the bound on the integral regardless of the tail of the interaction functions and show that the relevant quantity is only the matrix $\bm{H}$ of the $L^1$ norms of the interaction functions. The following lemma is key and is made possible due to the mass transportation equality.

\begin{lemma}\label{repartition of a cluster}
    Let $L>0$ and a $\Pois(\bm{h})$-cluster $G^m$ from $m\in\M$ and started at $t=0$. Then, for any $s\in[-L,0)$ and any $\xi\geq 0$,
    \[ \sum_{n=0}^{\infty} \left( \E{e^{\xi G^m([s+nL,s+ (n+1)L))}} -1 \right) \leq \E{e^{\xi G^m([0,\infty))}}-1.\]
\end{lemma}

\begin{proof}[Proof of \Cref{repartition of a cluster}]
    To prove this lemma, one just has to remark the following inequality. For any sequence $y_1,y_2,\cdots$ such that $y_n\geq 1$ for $n\geq 1$, we have
    \[ \sum_{n=1}^{\infty} (y_n-1) \leq -1 + \prod_{n=1}^{\infty} y_n.\]
    Indeed, it suffices to denote $x_n=y_n-1 \geq 0$ and expand the product.
        
    From this observation, we have
    \begin{align*}
        \sum_{n=0}^{\infty} \left( \E{e^{\xi G^m([s+nL,s+(n+1)L))\big)}} -1 \right) & = \E{ \sum_{n=0}^{\infty} \left( e^{\xi G^m([s+nL,s+(n+1)L))} -1 \right)}\\
        & \leq \E{ \prod_{n=0}^{\infty}  e^{\xi G^m([s+nL,s+(n+1)L))} }-1\\
        & = \E{e^{\xi G^m([0,\infty))}}-1.
    \end{align*}
This is the desired result.
\end{proof}

Going back to the proof of \Cref{moment exp}. Let us rewrite the integral as follows
\begin{equation}
    \itg_{-L}^{\infty} \big[g(t,m,\xi,L)-1\big]dt = \itg_{-L}^0 \sum_{n=0}^{\infty} \big[ g(t+nL,m,\xi,L)-1\big] dt.
\end{equation}
From \Cref{repartition of a cluster} we have that
\begin{equation*}
    \sum_{n=0}^{\infty} \big[ g(t+nL,m,\xi,L)-1\big] \leq \E{e^{\xi G^m([0,\infty))}}-1,
\end{equation*}
thus it follows that
\begin{equation}\label{n-1}
    \E{e^{\xi N(B)}} \leq \exp\left(  L \sum_{m\in M} \mu_m \E{e^{\xi  G^m([0,\infty))}-1} \right).
\end{equation}
It remains to control $\E{e^{ t G^m([0,\infty))}}$ with $t\geq 0$. At this point, the time embedding of $G^m$ is no more relevant since we have the equality in law 
\begin{equation}\label{cluster GW}
    G^m([0,\infty)) = \card(\T)
\end{equation}
where $\big((u,\tp(u))\big)_{u\in\T}$ is a $\Pois(\bm{H})$-Galton Watson tree with root of type $m$. The next lemma is the last piece of the proof.

\begin{lemma}\label{moment exp GW}
    Let $\big((u,\tp(u))\big)_{u\in\T}$ be a $\Pois(\bm{H})$-Galton Watson tree with root of any type. Suppose that $\bm{H}\in\Ge{r}{K}$ with $r\in(0,1)$ and $K\geq 0$, which means that
    \[ \Vvert \bm{H}^n \Vvert_{\infty} \leq K r^n, \quad n\in\N^*.\]
    Then we have,
    \[\log\left(\E{\exp\big(t \card(\T)\big)}\right) \leq  t \left[ 1+\frac{2K}{1-r}\right] = \frac{t}{\xi_{r,K}} \log\left(\dfrac{1+r}{2r}\right)\]
    for all $0\leq t\leq \xi_{r,K}$ with
    \[\xi_{r,K} = \dfrac{\log\left(\dfrac{1+r}{2r}\right)}{1+\dfrac{2 K}{1-r}}.\]
\end{lemma}
The proof of this result is postponed at the end of this paper.

Thus, up to \Cref{moment exp GW}, the proof is complete. Indeed from \Cref{moment exp GW} and \cref{cluster GW}, we have for $t=\xi=u\xi_{r,K}$,
\begin{equation}\label{n-2}
    \E{e^{\xi  G^m([0,\infty))}} \leq \left(\dfrac{1+r}{2r}\right)^{u}.
\end{equation}
Finally, injecting \cref{n-2} in \cref{n-1} concludes the proof of \Cref{moment exp}.

\subsection[Proof of Lemma 3.2]{Proof of \Cref{moment exp GW}}

It remains to prove \Cref{moment exp GW}. In \cite{suppDRR}, the authors proved a similar result, stating that a $\Pois(\bm{H})$-Galton Watson tree has exponential moments if the spectral norm of $\bm{H}$ is smaller than one. \Cref{moment exp GW} is stronger since we only need the spectral radius of $\bm{H}$ to be smaller than one (which is equivalent to require that $\bm{H}\in\Ge{r}{K}$ with some $r<1$). The proof is a bit more technical since we have to deal with the constant $K$, possibly $K>1$, from the property $\Ge{r}{K}$.

\begin{proof}[Proof of \Cref{moment exp GW}]
Let for $n\in\N$ and $m\in\M$,
\[g^m_n(t) = \log\left( \E{e^{t Z^m_n}}\right) \]
where $Z^m_n$ is the cardinal a $\Pois(\bm{H})$-Galton Watson tree with a root of type $m\in\M$ and clipped at generation $n$. 

The classical branching property of Galton Watson trees states the following. Given any point $u\in\T$, define the sub-tree $\T_{\geq u}$ as the sub-tree of $\T$ that contains $u$ and all its offspring (children, grandchildren and so on). Then we have, conditionally to $(u\in\T)$ and the type of $u$, the random variables $\big((v,\tp(v))\big)_{v\in\T_{\geq u}}$ and $\big((v,\tp(v))\big)_{v\in (\T \setminus \T_{\geq u}) \cup\{u\}}$ are independent and $\big((v,\tp(v))\big)_{v\in\T_{\geq u}}$ is distributed as a $\Pois(\bm{H})$-Galton Watson tree with root $u$ of type $\tp(u)\in\M$. 

Applying the branching property to all the points of the first generation (the children of the root) leads, for all $n\geq 0$, to the following equality in distribution
\[Z^m_{n+1} = 1 + \sum_{m'\in\M} \sum_{k=1}^{X_{m,m'}} Z^{m',k}_n\]
where $X_{m,m'} \sim \Pois(\bm{H}_{m,m'})$ and $Z^{m',k}_n \sim Z^{m'}_n$ and the random variables $X_{m,m'}, \ Z^{m',k}_n$ for $m'\in\M$, $k\geq 1$ are all independent. Thus we obtain that for all $m\in\M$,
\[ g_{n+1}^m(t) = t + \sum_{m'\in\M} \bm{H}_{mm'} (e^{g^{m'}_{n}(t)}-1).\]
One also has that
\[ g_{0}^m(t) = t.\]
Define $u_n^m(t) := g_{n+1}^m(t)-g_{n}^m(t) \geq 0$ and $u_n(t)$ the vector with coordinates $u_n^m(t)$ for $m\in\M$. Then
\[ u_0^m(t) = (e^t-1)\sum_{m'} \bm{H}_{mm'},\]
which can be written
\[ u_0(t) = (e^t-1)\bm{H} \bm{1}_{\M},\]
where $\bm{1}_{\M}$ is the vector with only ones. Using the relation between $g_n$ and $g_{n+1}$ gives
\begin{align*}
    u_{n+1}^m(t) & = \sum_{m'\in\M} \bm{H}_{mm'} (e^{g^{m'}_{n+1}(t)}-e^{g^{m'}_{n}(t)}) \\
    & \leq \sum_{m'\in\M} \bm{H}_{mm'} e^{g^{m'}_{n+1}(t)} (g^{m'}_{n+1}(t)-g^{m'}_{n}(t))\\
    & \leq \exp(\Vert g_{n+1}(t)\Vert_{\infty}) \sum_{m'\in\M} \bm{H}_{mm'} u^{m'}_{n}(t)\\
    & \leq \exp\left(t+\sum_{j=0}^n \Vert u_{j}(t)\Vert_{\infty}\right) \sum_{m'\in\M} \bm{H}_{mm'} u^{m'}_{n}(t),
\end{align*}
which becomes
\[ u_{n+1}(t) \leq \exp\left(t+\sum_{j=0}^n \Vert u_{j}(t)\Vert_{\infty}\right) \times \bm{H} \ u_{n}(t).\]
Unfolding the previous equation gives
\begin{align*}
    u_n(t) & \leq \exp\left(\sum_{i=1}^n \left[t+\sum_{j=0}^{i-1} \Vert u_{j}(t)\Vert_{\infty}\right]\right) \times \bm{H}^n  u_{0}(t)\\
    & = (e^t-1) \exp\left(\sum_{i=1}^n \left[t+\sum_{j=0}^{i-1} \Vert u_{j}(t)\Vert_{\infty}\right]\right) \times \bm{H}^{n+1} \bm{1}_{\M}.
\end{align*} 
Let $0\leq t \leq \log(1/r)$ and constants $K_1>0$, $r<r_1<1$ that will be fixed later. We will prove by induction on $n$ that $\Vert u_{n}(t)\Vert_{\infty} \leq t K_1 r_1^n$ if $t$ is small enough (the precise condition is given at the end).

\underline{$n=0:$} Since $t\leq \log(1/r)$, we have $e^t-1\leq t/r$ and by $\Ge{r}{K}$ we have $\Vert \bm{H} \bm{1}_{\M}\Vert_{\infty} \leq Kr^1 \times 1 = K r$. Thus if $K_1\geq K$ it works.

\underline{$0,1,\cdots,n-1 \to n:$} In this case we have on the one hand
\[ \sum_{i=1}^n \left[t+\sum_{j=0}^{i-1} \Vert u_{j}(t)\Vert_{\infty}\right] \leq n t \left( 1+\dfrac{K_1}{1-r_1}\right).\]
On the other hand by $\Ge{r}{K}$,
\[ \Vert \bm{H}^{n+1} \bm{1}_{\M}\Vert_{\infty} \leq K r^{n+1}.\]
Thus, since $e^t-1\leq t/r$,
\[\Vert u_{n}(t)\Vert_{\infty} \leq t K  \exp\left(n\left[ \log(r)+t\Big(1+\frac{K_1}{1-r_1}\Big)\right] \right).\]
The proof of the induction is complete if 
\begin{equation*}
    \left\{
    \begin{aligned}
        & K_1 \geq K \\
        & \log(r)+t\Big(1+\frac{K_1}{1-r_1}\Big) \leq \log(r_1) \ \Leftrightarrow \ t\leq \frac{\log(r_1/r)}{1+\frac{K_1}{1-r_1}} 
    \end{aligned}
    \right.
\end{equation*}
Thus we set $K_1 = K$ and arbitrarily choose $r_1 = (1+r)/2$. Thus the above bound for $t$ is exactly $\xi_{r,K}$, leading to, for all $t\leq \xi_{r,K}$, for all $n\geq 0$, we have $\Vert u_{n}(t)\Vert_{\infty} \leq t K r_1^n$. Going back to $g_n(t)$ leads to
\[ \forall m\in\M, \ \forall n\in\N, \ \forall t\leq \xi_{r,K}, \ g^m_n(t) \leq t\Big(1+\frac{K}{1-r_1}\Big).\]
Letting $n\to\infty$ gives the intended result by monotone convergence theorem.
\end{proof}

\subsection[Proof of Theorem 1.6]{Proof of \Cref{functional}}

By following the steps of the proof of \Cref{moment exp} one obtains
\begin{equation*}
    \E{e^{\xi N(f)}} = \exp\left( \sum_{m\in \M} \itg_{-\infty}^{\infty} \big[g(t,m,\xi,f)-1\big] \mu_m dt\right)
\end{equation*}
where
\[g(t,m,\xi,f) = \E{e^{\xi G^m(f\circ \tau_t)}}\]
and $G^m$ is a cluster born at time $t=0$ and with root of type $m$ and $\tau_t(x) = x+t, \ x\in\R$.
Then we discretize the integral at scale $T$,
\begin{equation*}
    \itg_{-\infty}^{\infty} \big[g(t,m,\xi,f)-1\big]  dt = \itg_{0}^{T} \sum_{n\in\Z} \big[g(t+nT,m,\xi,f)-1\big] dt.
\end{equation*}
The same argument as the one for the proof of \Cref{repartition of a cluster} leads to
\begin{equation*}
    \sum_{n\in\Z} \big[g(t+nT,m,\xi,f)-1\big] \leq \E{e^{\xi G^m(\sum_{n\in\Z} f\circ \tau_{t+nT})}}-1.
\end{equation*}
Since $\sum_{n\in\Z} f\circ \tau_{t+nT}\leq \vert f\vert^{\infty}_{1,T}$, it follows that
\begin{equation*}
    \E{e^{\xi N(f)}} \leq \exp\left( T \sum_{m\in \M} \mu_m\E{e^{\xi \vert f\vert^{\infty}_{1,T} G^m(\R)}-1} \right).
\end{equation*}
\Cref{moment exp GW} applied with $t=\xi \vert f\vert^{\infty}_{1,T}$ concludes the proof of \Cref{functional}.

\vspace{0.5\baselineskip}

\textbf{Acknowledgments.} I am grateful to my PhD supervisors, Vincent Rivoirard and Patricia Reynaud Bouret, for their advice and helpful comments. I would also like to thank the referee for his/her comments and pertinent questions, which have enabled me to significantly improve this paper.


\begin{thebibliography}{99}

\bibitem{PPQ}
Pierre Brémaud.
\newblock {\em Point Processes and Queues}.
\newblock Springer New York, 1981.

\bibitem{BréMas}
Pierre Brémaud and Laurent Massouli{\'e}.
\newblock Stability of nonlinear hawkes processes.
\newblock {\em The Annals of Probability}, pages 1563--1588, 1996.

\bibitem{daley2003introduction}
Daryl~J Daley, David Vere-Jones, et~al.
\newblock {\em An introduction to the theory of point processes: volume I:
  elementary theory and methods}.
\newblock Springer, 2003.

\bibitem{suppDRR}
Sophie Donnet, Vincent Rivoirard, and Judith Rousseau.
\newblock Supplement to “nonparametric bayesian estimation for multivariate
  hawkes processes”.
\newblock {\em The Annals of Statistics}, 48(5):2698 -- 2727, 2020.

\bibitem{GalvesLoc}
Antonio Galves and Eva L{\"o}cherbach.
\newblock Infinite systems of interacting chains with memory of variable
  length—a stochastic model for biological neural nets.
\newblock {\em Journal of Statistical Physics}, 151:896--921, 2013.

\bibitem{LassoHawkes}
Niels~Richard Hansen, Patricia Reynaud-Bouret, and Vincent Rivoirard.
\newblock Lasso and probabilistic inequalities for multivariate point
  processes.
\newblock {\em Bernoulli}, 21(1):83--143, 2015.

\bibitem{Hawkes71}
Alan~G. Hawkes.
\newblock Spectra of some self-exciting and mutually exciting point processes.
\newblock {\em Biometrika}, 58:83--90, 1971.

\bibitem{HawkesOakes}
Alan~G Hawkes and David Oakes.
\newblock A cluster process representation of a self-exciting process.
\newblock {\em Journal of applied probability}, 11(3):493--503, 1974.

\bibitem{Khezeli}
Ali Khezeli.
\newblock Unimodular random measured metric spaces and palm theory on them.
\newblock 2023.

\bibitem{PPK}
J~F~C Kingman.
\newblock {\em {Poisson Processes: Oxford Studies In Probability. 3}}.
\newblock Oxford University Press, 12 1992.

\bibitem{Otaga}
Yosihiko Ogata.
\newblock Statistical models for earthquake occurrences and residual analysis
  for point processes.
\newblock {\em Journal of the American Statistical association}, 83(401):9--27,
  1988.

\bibitem{AgeDepHP}
Mads~Bonde Raad, Susanne Ditlevsen, and Eva L{\"o}cherbach.
\newblock {Stability and mean-field limits of age dependent Hawkes processes}.
\newblock {\em Annales de l'Institut Henri Poincaré, Probabilités et
  Statistiques}, 56(3):1958--1990, 2020.

\bibitem{RBRoy}
Patricia Reynaud-Bouret and Emmanuel Roy.
\newblock Some non asymptotic tail estimates for hawkes processes.
\newblock {\em Bulletin of the Belgian Mathematical Society-Simon Stevin},
  13(5):883--896, 2007.

\bibitem{ROUEFF20161710}
François Roueff, Rainer {von Sachs}, and Laure Sansonnet.
\newblock Locally stationary hawkes processes.
\newblock {\em Stochastic Processes and their Applications}, 126(6):1710--1743,
  2016.

\end{thebibliography}
\end{document}